\documentclass[11pt]{article}
\usepackage{amsmath,amsfonts,amsthm,amssymb, mathtools,makecell,setspace}
\usepackage{mathrsfs, graphicx,color,latexsym, tikz, calc, float,hyperref
}
\usepackage{enumerate}

\usepackage{float}

\usetikzlibrary{shadows}
\usetikzlibrary{patterns,arrows,decorations.pathreplacing}
\def\b{\beta}
\def\a{\alpha}

\newcommand\Cay{\mathrm{Cay}}

\newcommand\V{\varepsilon}

\newtheorem{theorem}{Theorem}[section]

\newtheorem{lemma}{Lemma}[section]

\newtheorem{remark}{Remark}[section]
\newtheorem{prop}{Proposition}[section]

\allowdisplaybreaks

\title{\bf \Large  On the number of  spanning trees of bicirculant graphs \footnote{  E-mail addresses: yjlisy@163.com (J. Yang, corresponding author), 
		18187166242@163.com (F. Xian).}}
\author{
	{\small  Jing Yang$\dagger$, Fangming Xian}\\[2mm]
	\small  School of Statistics and Mathematics, Yunnan University of Finance and Economics \\
	\small   Kunming, Yunnan, 650000, China.\\}
 
\begin{document}
\maketitle
\begin{abstract}
 A bi-Cayley graph over a cyclic group $\mathbb{Z}_n$ is called a bicirculant graph. Let
 $\Gamma=BC(\mathbb{Z}_n; R,T,S)$ be a bicirculant graph with $R=R^{-1}\subseteq \mathbb{Z}_n\setminus \{0\}$ and $T=T^{-1}\subseteq \mathbb{Z}_n\setminus \{0\}$ and $S\subseteq \mathbb{Z}_n$. In this paper, using Chebyshev polynomials, we obtain a closed formula
  for the number of spanning trees of bicirculant graph $\Gamma$, investigate some arithmetic properties of the number of  spanning trees of $\Gamma$, and find its asymptotic behaviour as $n$ tends infinity. In addition,
  we show that $F(x)=\sum_{n=1}^{\infty}\tau(\Gamma)x^n$ is a rational function with integer coefficients.\\


\end{abstract}

\section{Introduction}
We only consider simple undirected graphs in this paper. Let $\Gamma=(V(\Gamma),E(\Gamma))$ be a  graph with vertex set $V(\Gamma)$ and edge set $E(\Gamma)$. The \textit{adjacency matrix} of $\Gamma$ is a $V(\Gamma)\times V(\Gamma)$ matrix where its rows
and columns indexed by the vertex set of $\Gamma$ and its $(u, v)$-entry is $1$ if the vertices $u$ and $v$ are adjacent
and $0$ otherwise. By eigenvalues of $\Gamma$ we mean the eigenvalues of its adjacency matrix.
The \textit{Laplacian matrix} of $\Gamma$ is $L(\Gamma) := D(\Gamma)-A(\Gamma)$
where $A(\Gamma)$ is the adjacency matrix of $\Gamma$, and
$D(\Gamma)=\operatorname{diag}(d_{1},\dots ,d_{n})$ with $d_{i}$ denoting the degree of the vertex $i$. Since $L(\Gamma)$ is symmetric and positive semidefinite, its eigenvalues are real and nonnegative
and can be ordered as $0 = \lambda_{1} \le \lambda_{2} \le \dots \le \lambda_{|V(\Gamma)|}$.

A \textit{tree} is a connected undirected graph without cycles. A \textit{spanning tree} in a graph $\Gamma$ is a subgraph that is a tree containing all the vertices of $\Gamma$. The \textit{complexity} of a finite connected graph $\Gamma$, denoted by $\tau(\Gamma)$, is the number of spanning trees of $\Gamma$. The number of spanning trees in a graph is an important invariant, and it also serves as a crucial indicator of network reliability. A classic result known as Kirchhoff’s Matrix Tree Theorem\cite{tree} states that the number of spanning trees
in a graph can be expressed as the product of its non-zero Laplacian eigenvalues divided
by the number of vertices. In practice, though,
this method of counting spanning trees by the product of eigenvalues of the
Laplacian matrix is infeasible for large graphs. For some special classes of graphs, it is possible to give explicit, simple formulae for the number of spanning trees, one can see
\cite{AlmostM1,AlmostM2,M1,wheel, fan,prism,ladder,ladder1,latt,ati,ati1,sie1,grid}.

Cayley graph is a class of highly symmetric graphs. In algebraic
graph theory, the study of Cayley graphs over finite groups is one of the prosperous topics and
attracts tremendous amount of attention from the literature. Let $G$ be a group and $S$ be a subset of $G$ such that $S^{-1}= S$ and $1\notin S$, the Cayley graph $\Gamma=\mathrm{Cay}(G, S)$ over
$G$ with respect to $S$ is a graph with $V(\Gamma)=G$ and 
$E(\Gamma) = \{{g, sg}~|~ g\in G, s\in S\}$. Let $R$, $L$ and $S$ be subsets of a group $G$ such that $R=R^{-1}$, $T=T^{-1}$ and  
$R\cup T$ does not contain the identity element of $G$.  
Define the bi-Cayley graph $\Gamma=BC(G;R,T,S)$ to have vertex set the union of the right part  
$G_0=\{h_0\mid h\in G\}$ and the left part  $G_1=\{h_1\mid h\in G\}$,
and edge set the union of the right edges  
$\{\{h_0,g_0\}\mid gh^{-1}\in R\},$ 
the left edges  
$\{\{h_1,g_1\}\mid gh^{-1}\in T\},$  
and the spokes  
$\{\{h_0,g_1\}\mid gh^{-1}\in S\}$. 
Clearly, $|S|>0$ if bi-Cayley graph $\Gamma$ is connected.
A Cayley (resp. bi-Cayley) graph on
a cyclic group is called a \textit{circulant graph} (resp. \textit{bicirculant graph}).

Let
$$
P(x)=a_0+a_1x+\cdots+a_sx^s=a_s \prod_{k=1}^s\left(x-\alpha_k\right)
$$
 be a nonconstant polynomial with complex coefficients.
The  \textit{Mahler measure} \cite{mah} of $P(x)$ is defined to be
\begin{equation}
M(P):=\exp \left(\int_0^1 \log \left|P\left(e^{2 \pi\mathbf{i} t}\right)\right| d t\right).
\end{equation}
Actually, an alternative form of $M(P)$ has appeared in \cite{leh}.
 That is,
\begin{equation}
M(P)=\left|a_s\right| \prod_{\left|\alpha_i\right|>1}\left|\alpha_i\right|.
\end{equation}
The concept of Mahler measure can be naturally extended to the class of \textit{Laurent polynomials}
$$P(x)=a_0 x^t+a_1 x^{t+1}+\cdots+a_{s-1} x^{t+s-1}+a_s x^{t+s}=a_s x^t \prod_{i=1}^s\left(x-\alpha_i\right),$$
 where $a_s \neq 0$, and $t$ is an arbitrary but not necessarily positive integer.

Recently, Mednykh et al. \cite{MMA1,med} developed a new method to produce explicit formulas for the number of spanning trees of
circulant graphs in light of Chebyshev polynomials. Moreover, they investigated their arithmetic properties and asymptotic behaviour.
Finally, they considered the generating function for the number of spanning trees of circulant graphs.
Using the techniques developed in \cite{MMA1,med}, a similar work by Hua et al. \cite{hua} for Cayley graph over dihedral group.

Inspired by the work of Mednykh et al. \cite{MMA1,med,hua}, in this paper, we obtain a closed formula for the number of spanning
trees of bicirculant graphs, denoted by $\tau(\Gamma)$, in light of Chebyshev polynomials, and investigate their arithmetic properties and asymptotic
behaviour. Moreover, the generating function for the number of spanning trees of bicirculant graphs is considered.

The structure of the paper is as follows. In Section 2, some basic definitions and preliminary results are given.
In Section 3, using Chebyshev polynomials, we present explicit formulas for the number of spanning trees of bicirculant graphs.
 In Section 4, we provide some arithmetic properties of the number of spanning trees. 
 In Section 5, we use explicit formulas for the complexity in order to produce its asymptotic in light of Mahler measure of the associated
polynomials. In Section 6, we show that $F(x)=\sum_{n=1}^{\infty}\tau(\Gamma)x^n$ is a rational function with integer coefficients satisfying some symmetry property. In Section 7, we illustrate our results by a series of examples.
\section{Preliminaries}
In this section, we list some definitions and known results which are used in this paper.
First, we introduce the definition of the resultant of two polynomials.
Let $f(x)=a_0x^n+a_1x^{n-1}+\cdots+ a_n$ and $g(x)=b_0x^m+b_1x^{m-1}+\cdots+ b_m$ be two polynomials with $a_i,b_j\in \mathbb{R}$.
Then the \textit{resultant} of $f(x)$ and $g(x)$ is
$$
\mathrm{Res}(f(x), g(x))=\det\left(\begin{array}{ccccccccc}a_0 & a_1 & a_2 & \cdots & \cdots & a_n & 0 & \cdots & 0 \\ 0 & a_0 & a_1 & \cdots & \cdots & a_{n-1} & a_n & \cdots & 0 \\ 0 & 0 & a_0 & \cdots & \cdots & a_{n-2} & a_{n-1} & \cdots & 0 \\ \vdots & \vdots & \vdots & \vdots & \vdots & \vdots & \vdots & \vdots & \vdots \\ 0 & 0 & \cdots & 0 & a_0 & \cdots & \cdots & \cdots & a_n \\ b_0 & b_1 & b_2 & \cdots & \cdots & \cdots & b_m & \cdots & 0 \\ 0 & b_0 & b_1 & \cdots & \cdots & \cdots & b_{m-1} & b_m & \cdots \\ \vdots & \vdots & \vdots & \vdots & \vdots & \vdots & \vdots & \vdots & \vdots \\ 0 & \cdots & 0 & b_0 & b_1 & \cdots & \cdots & \cdots & b_m\end{array}\right).
$$
\begin{lemma}\cite{book}\label{book}
Let $f(x)=a_0x^n+a_1x^{n-1}+\cdots+ a_n$ and $g(x)=b_0x^m+b_1x^{m-1}+\cdots+ b_m$ be two polynomials with $a_i,b_j\in \mathbb{R}$.
Let $x_1,x_2,\ldots,x_n$ and $z_1,z_2,\ldots,z_m$ be the roots of $f(x)=0$ and $g(x)=0$, respectively.
Then

\begin{itemize}
\item[(1)]$\mathrm{Res}(f(x), g(x))=a_0^mb_0^n\prod_{j=1}^m\prod_{i=1}^n(x_i-z_j)$,
\item[(2)] $\mathrm{Res}(f(x), g(x)+h(x)f(x))=\mathrm{Res}(f(x),g(x))$ for any polynomial $h(x)$,
\item[(3)] $\mathrm{Res}(f(x), g_1(x)g_2(x))=\mathrm{Res}(f(x), g_1(x))\cdot \mathrm{Res}(f(x), g_2(x))$, where $g(x)= g_1(x)g_2(x)$.
\end{itemize}
\end{lemma}

Let $\mathbb{F}$ be a number field, and $\mathbb{F}^{m\times n}$ be the set of $m\times n$ matrices.
\begin{lemma}\label{mat}\cite{GX}
 Let $A, B, C, D\in\mathbb{F}^{n\times n}$ with $AC=CA$. Then
 $$
\det\begin{pmatrix}
 A&B\\
 C&D\
 \end{pmatrix}
=\det(AD-CB).
$$
\end{lemma}

The following equation is the famous \textit{Lagrange's identity}.
\begin{lemma}\label{lag}
Let $a_1,a_2,\ldots,a_n$, $b_1,b_2,\ldots,b_n\in \mathbb{R}$. Then
$$\left(\sum_{j=1}^n a_j^2\right)\left(\sum_{j=1}^n b_j^2\right)-\left(\sum_{j=1}^n a_jb_j\right)^2=\sum_{1\le j<i\le n}(a_jb_i-a_ib_j)^2.$$
\end{lemma}

	An $n\times n$ \textit{circulant matrix} $C$ have the form
$$C=
\begin{pmatrix}
	c_0 & c_1 & c_2 & \cdots & c_{n-1} \\
	c_{n-1} & c_0 & c_1 & \cdots & c_{n-2} \\
	c_{n-2} & c_{n-1} & c_0 & \cdots & c_{n-3} \\
	\vdots & \vdots & \vdots & \ddots & \vdots \\
	c_1 & c_2 & c_3 & \cdots & c_0
\end{pmatrix},
$$
where each row is a cyclic shift of the row above it. We see that the $(j,k)$-th entry of $C$ is $c_{k-j} \pmod{n}$. We denote such a circulant matrix by $C=\mathrm{Circ}(c_0,c_1,\ldots , c_{n-1})$. The eigenvalues of $C$ are  $\lambda_j=g(\V_{n}^j)$, $j=0,1,2,\ldots,n-1$,
where $\V_{n}=\exp(2\pi \textit{i}/n)$  and $g(x)=\sum_{k=0}^{n-1} c_k x^{kj}$.
Let $\Lambda=\Cay(\mathbb{Z}_n,\{\pm s_1,\pm s_2,\ldots,\pm s_\ell\})$ be a circulant graph.
Note that the vertex $i$ is adjacent to the vertices $i\pm s_1,i\pm s_2,\ldots, i\pm s_\ell~(\bmod~n)$.
Then the adjacency matrix of $\Lambda$ is $\sum_{j=1}^{\ell}\left(Q^{s_j}+Q^{-s_j}\right)$,
where $Q=\operatorname{circ}(\underbrace{0,1,0,\ldots,0}_{n})$.
Let $\Gamma=BC(\mathbb{Z}_n;R,T,S)$ be a bicirculant graph. By \cite[Lemma 3.1]{GX},
the adjacency matrix of $\Gamma$ is
$D=\begin{pmatrix}
A & C \\
C^{\top} & B
\end{pmatrix}$,
where $A, B, C$ are the adjacency matrices of $\Cay(\mathbb{Z}_n,R), \Cay(\mathbb{Z}_n,T), \Cay(\mathbb{Z}_n,S)$, respectively,
and $C^{\top}$ means the transpose of $C$. Based on the above facts, we can obtain the following result directly.

\begin{lemma}\label{adj}
Let $R=R^{-1}=\{\ell_1, \ell_2\ldots, \ell_r\} \subseteq \mathbb{Z}_n\setminus \{0\}$,
$T=T^{-1}=\{m_1, m_2,\ldots, m_t\}\subseteq \mathbb{Z}_n \setminus \{0\}$ and
$S=\{u_1,u_2\ldots, u_s\}\subseteq \mathbb{Z}_n$.
Then the adjacency matrix of $\Gamma=BC(\mathbb{Z}_n; R,S,T)$ is
$$
A(\Gamma)=\begin{pmatrix}
 \sum_{j=1}^{r}Q^{\ell_j} & \sum_{j=1}^{s}Q^{u_j}\\
 \sum_{j=1}^{s}Q^{-u_j} & \sum_{j=1}^{t}Q^{m_j}
 \end{pmatrix},
$$
where $Q=\operatorname{circ}(\underbrace{0,1,0,\ldots,0}_{n})$.
\end{lemma}

Let $T_n(w) = \cos(n\arccos w)$ be the Chebyshev polynomial of the first kind.
The following lemma provides basic properties of $T_n(w)$.

\begin{lemma}\label{chee}\cite{med}
 Let $w=\frac{1}{2}\left(z+\frac{1}{z}\right)$ and $T_n(w)$ be the Chebyshev polynomial of the first kind. Then
 $$T_n(w)=\frac{1}{2}\left(z^n+\frac{1}{z^n}\right),$$
 where $z\in \mathbb{C}\setminus \{0\}$ and $n$ is a positive integer.
\end{lemma}

\begin{lemma}\cite{YS}\label{ch}
Let $H(z)=\prod_{\ell=1}^k(z-z_\ell)(z-z_\ell^{-1})$ and $H(1)\neq 0$. Then
$$\prod_{j=1}^{n-1}H(\varepsilon_n^j)=\prod_{\ell=1}^k\frac{T_n(w_\ell)-1}{w_\ell-1},$$
where $w_\ell=\frac{1}{2}(z_{\ell}+z_{\ell}^{-1})$ for $\ell=1,\ldots,k$ and $T_{n}(w)$ is the Chebyshev polynomial of the first kind.
\end{lemma}

Let $P(z)$ be a polynomial of degree $k$ with integer coefficients and $P(z)=P(z^{-1})$. Then $P(z)$ have the following form
$$
P(z)=\eta_0+\sum_{j=1}^k\eta_j(z^j+z^{-j}),
$$
where $\eta_0,\eta_1,\ldots,\eta_k$ are integers.
Let $w=\frac{1}{2}\left(z+z^{-1}\right)$ and $K(w)=\eta_0+\sum_{j=1}^k2\eta_jT_j(w)$. By Lemma \ref{chee}, we have $P(z)=K(w)$.
The polynomial $K(w)$ is called the \textit{Chebyshev transform} of $P(z)$.

\begin{lemma}\label{che}
	Let $C(z)=\eta_0+\sum_{j=1}^k\eta_j(z^j+z^{-j})$ and $1$ be a root of $C(z)$ 
	with multiplicity 2. Then
	$$\prod_{j=1}^{n-1}C(\V_n^j)=\frac{1}{C^{\prime\prime}(1)}(-1)^{nk-1}\eta_k^{n}n^22^k\prod_{\ell=1}^{k-1}(T_n(w_\ell)-1),$$
	where $C^{\prime\prime}(z)$ denotes $2$-order derivative of $C_1(z)$ and $w_1, w_2,\ldots, w_k$ are the roots, different from 1, of Chebyshev transform of $C(z)=0$.
\end{lemma}
\begin{proof}
	Let $C_1(z)=\frac{z^k}{\eta_k}C(z)$. Then $C_1(z)$ is a monic polynomial with the same roots as $C(z)$ and the degree of $C_1(z)$ is $2k$. 
	Suppose that the roots of $C_1(z)=0$ are $1,1,z_1,z_1^{-1},\ldots, z_{k-1},z_{k-1}^{-1}$ and $H(z)=\prod_{j=1}^{k-1}(z-z_j)(z-z_j^{-1})$.
	Then $C(z)=\frac{\eta_k(z-1)^2}{z^k}H(z)$.
	Since $H(1)\not=0$, by Lemma \ref{ch}, we have
	$$\prod_{j=1}^{n-1}H(\varepsilon_n^j)=\prod_{\ell=1}^{k-1}\frac{T_n(w_\ell)-1}{w_\ell-1},$$
	where $w_\ell=\frac{1}{2}(z_\ell+z_\ell^{-1})$ for $\ell=1,2,\ldots,k-1$.
	Note that 
	$$
	\prod_{j=1}^{n-1}(1-\varepsilon_n^j)=\lim_{z\to1}\prod_{j=1}^{n-1}(z-\varepsilon_n^j)=\lim_{z\to1}\frac{z^n-1}{z-1}=n.
	$$
	Therefore,
	$$
	\prod_{j=1}^{n-1}C(\V_n^j)=\prod_{j=1}^{n-1}\frac{\eta_k(\V_n^j-1)^2}{\V_n^{jk}}\prod_{j=1}^{n-1}H(\V_n^j)=
	(-1)^{(n+1)k}\eta_k^{n-1}n^2\prod_{\ell=1}^{k-1}\frac{T_n(w_\ell)-1}{w_\ell-1}.
	$$
	Note that the Taylor expansion of the polynomial $C_1(z)$ at $z = 1$ is
	$$
	C_1(z)=C_1(1)+C_1^{\prime}(1)(z-1)+\frac{C_1^{\prime\prime}(1)(z-1)^2}{2!}+
	\frac{C_1^{\prime\prime\prime}(1)(z-1)^3}{3!}+\cdots.
	$$
	Since $C_1(1)=C_1^{\prime}(1)=0$, by Lemma \ref{book}, we have
	$$
	\begin{aligned}
		(-1)^{k-1}2^{k-1}\prod_{\ell=1}^{k-1}(w_\ell-1)&=\prod_{\ell=1}^{k-1}(z_\ell-1)(z_\ell^{-1}-1)=\prod_{\substack{C_1(z)=0,\\z\not=1}}(z-1)\\
		&=\mathrm{Res}(z-1,\frac{C_1(z)}{(z-1)^2})\\
		&=\mathrm{Res}(z-1,\frac{C_1^{\prime\prime}(1)}{2}+\frac{C_1^{\prime\prime\prime}(1)(z-1)}{6}+\cdots)\\
		&=\frac{C_1^{\prime\prime}(1)}{2}.
	\end{aligned}
	$$
	Note that $C_1^{\prime\prime}(1)=\frac{C^{\prime\prime}(1)}{\eta_k}$. Therefore, we have
	$$
	\prod_{j=1}^{n-1}C(\V_n^j)=\frac{1}{C^{(2)}(1)}(-1)^{nk-1}\eta_k^{n}n^22^k\prod_{\ell=1}^{k-1}(T_n(w_\ell)-1).
	$$
	This completes the proof.
	
\end{proof}

Now we define
\begin{itemize}
	\item[] $R_1=\{ \pm \a_1, \pm \a_2,\ldots, \pm \a_r\}$,~~$T_1=\{\pm \beta_1, \pm \beta_2,\ldots, \pm \beta_t\}$,~~
	$S_1=\{\gamma_1, \gamma_2,\ldots, \gamma_s\}$,
	\item[] $R_2=\{ \pm \a_1, \pm \a_2,\ldots, \pm \a_r, \frac{n}{2}\}$,~~ $T_2=\{\pm \beta_1, \pm \beta_2,\ldots, \pm \beta_t\}$,~~
	$S_2=\{\gamma_1, \gamma_2,\ldots, \gamma_s\}$,
	\item[] $R_3=\{ \pm \a_1, \pm \a_2,\ldots, \pm \a_r\}$, ~~$T_3=\{\pm \beta_1, \pm \beta_2,\ldots, \pm \beta_t,\frac{n}{2}\}$,~~
	$S_3=\{\gamma_1, \gamma_2,\ldots, \gamma_s\}$,
	\item[] $R_4=\{ \pm \a_1, \pm \a_2,\ldots, \pm \a_r, \frac{n}{2}\}$, ~~$T_4=\{\pm \beta_1, \pm \beta_2,\ldots, \pm \beta_t,\frac{n}{2}\}$,~~
	$S_4=\{\gamma_1, \gamma_2,\ldots, \gamma_s\}$,
\end{itemize}
where $1\le \a_1<\cdots <\a_r<\frac{n}{2}$, $1\le \b_1<\cdots<\b_t<\frac{n}{2}$, $0\le \gamma_1<\cdots<\gamma_s\le n-1$.

Throughout the rest of this paper, we always assume that
$$
\Gamma_j=BC(\mathbb{Z}_n;R_j,T_j,S_j),
$$
where $j=1,2,3,4$. Obviously, any bicirculant graph $\Gamma$ must be one of $\Gamma_1,\Gamma_2, \Gamma_3, \Gamma_4$
and $n$ is even for $\Gamma_2, \Gamma_3, \Gamma_4$. 
To ensure the graph $\Gamma$ is connected, we further assume that $\Gamma$ satisfies 
at least one of the following conditions:
\begin{itemize}
\item[(a)]
$\gcd(n,\a_1,\a_2,\ldots, \a_r)=1$ and $s>0$,
\item[(b)]
$\gcd(n,\b_1,\b_2,\ldots, \b_t)=1$ and $s>0$,
\item[(c)]
$\gcd(n,\gamma_\ell-\gamma_k, 0\le\ell<k\le n-1)=1$.
\end{itemize}

Let 
$$
\begin{aligned}
	\mathcal{A}(z)&=2r+s-\sum_{j=1}^r\left(z^{\a_j}+z^{-\a_j}\right),\\
	\mathcal{B}(z)&=2t+s-\sum_{j=1}^t\left(z^{\b_j}+z^{-\b_j}\right),\\
	\mathcal{C}(z)&=\sum_{j=1}^s z^{\gamma_j}. 
\end{aligned}
$$
Define
\begin{align}
	P_1(z)&=\mathcal{A}(z)\mathcal{B}(z)-\mathcal{C}(z^{-1})\mathcal{C}(z),\\
	P_{2}(z)&=(\mathcal{A}(z)+2)\mathcal{B}(z)-\mathcal{C}(z^{-1})\mathcal{C}(z),\\
	P_{3}(z)&=\mathcal{A}(z)(\mathcal{B}(z)+2)-\mathcal{C}(z^{-1})\mathcal{C}(z), \\
	P_{4}(z)&=(\mathcal{A}(z)+2)(\mathcal{B}(z)+2)-\mathcal{C}(z^{-1})\mathcal{C}(z).
\end{align}
Clearly, $P_1(z),P_2(z),P_3(z),P_4(z)$ have same degree and $P_j(z)=P_j(z^{-1})$ for $j=1,2,3,4$.
In the following, we always assume that the degree of $P_j(z)$ is $k$ for $j=1,2,3,4$.

	\begin{lemma}\label{one}
	Let $\Gamma$ be a connected bicirculant graph such that at least $\gcd(\a_1,\a_2,\ldots, \a_r)=1$,  $\gcd(\b_1,\b_2,\ldots, \b_t)=1$, $\gcd(\gamma_\ell-\gamma_k, 0\le\ell<k\le n-1)=1$ hold. Then 
	$P_1(e^{\textbf{i}\theta})=0 $ if and only if $e^{\textbf{i}\theta}=1$, and $P_j(e^{\textbf{i}\theta})>0$ for $j=2,3,4$.
\end{lemma}
\begin{proof}
	Since $\mathcal{A}(z)=2r+s-\sum_{j=1}^r\left(z^{\a_j}+z^{-\a_j}\right)$, 
	$\mathcal{B}(z)=2t+s-\sum_{j=1}^t\left(z^{\b_j}+z^{-\b_j}\right)$, and
	$\mathcal{C}(z)=-\sum_{j=1}^s z^{\gamma_j}$,
	we have
	$$
	\begin{aligned}
		\mathcal{A}(e^{\mathbf{i}\theta})
		&=2r+s-\sum_{j=1}^r \left(e^{\mathbf{i}\theta\a_j}+e^{-\mathbf{i}\theta\a_j}\right)
		=s+2\sum_{j=1}^r(1-\cos(\theta \a_j))\ge s,\\
	\mathcal{B}(e^{\mathbf{i}\theta})
	&=2r+s-\sum_{j=1}^t \left(e^{\mathbf{i}\theta\b_j}+e^{-\mathbf{i}\theta\b_j}\right)
	=s+2\sum_{j=1}^t(1-\cos(\theta \b_j))\ge s,\\
		|\mathcal{C}(e^{\mathbf{i}\theta})|^2&=\mathcal{C}(e^{\mathbf{i}\theta})\mathcal{C}(e^{\mathbf{-i}\theta})=\left(\sum_{\ell=1}^{s} {e^{\mathbf{i}\theta\gamma_\ell}}\right)\left(\sum_{k=1}^{s} {e^{\mathbf{-i}\theta\gamma_k}}\right)
		=\sum_{1\le \ell,k\le s}\cos(\theta(\gamma_\ell-\gamma_k)) \le s^2.
	\end{aligned}
	$$
	Therefore, we obtain
	$$
	\begin{aligned}
		P_1(e^{\textbf{i}\theta})&=\mathcal{A}(e^{\mathbf{i}\theta})\mathcal{B}(e^{\mathbf{i}\theta}) -|\mathcal{C}(e^{\mathbf{i}\theta})|^2
		\ge s^2-s^2=0\\
		P_2(e^{\textbf{i}\theta})&=(\mathcal{A}(e^{\mathbf{i}\theta})+2)\mathcal{B}(e^{\mathbf{i}\theta}) -|\mathcal{C}(e^{\mathbf{i}\theta})|^2\ge s^2-s^2+2s=2s,\\
	P_3(e^{\textbf{i}\theta})&=\mathcal{A}(e^{\mathbf{i}\theta})(\mathcal{B}(e^{\mathbf{i}\theta})+2) -|\mathcal{C}(e^{\mathbf{i}\theta})|^2\ge s^2-s^2+2s=2s,\\
	P_4(e^{\textbf{i}\theta})&=(\mathcal{A}(e^{\mathbf{i}\theta})+2)(\mathcal{B}(e^{\mathbf{i}\theta})+2)-|\mathcal{C}(e^{\mathbf{i}\theta})|^2\ge s^2-s^2+4s=4s.
	\end{aligned}
	$$
	Note that $P_1(e^{\textbf{i}\theta})=0$ if and only if $\cos(\theta \a_i)=\cos(\theta \b_j)=\cos(\theta(\gamma_\ell-\gamma_k))=1$ for $1\le i\le r$ and  $1\le j\le t$ and
	 $1\le \ell,k\le s$.
	If $\cos(\theta \a_i)=\cos(\theta \b_j)=\cos(\theta(\gamma_\ell-\gamma_k))=1$ for $1\le i\le r$ and  $1\le j\le t$ and
	$1\le \ell,k\le s$, then $\theta \a_i=2h_i\pi$ and $\theta \b_j=2n_j\pi$ 
	and $\theta(\gamma_\ell-\gamma_k)=2m_{\ell,k}\pi$, where $h_i, n_j, m_{\ell,k}$ are integers.
	Since
 $\gcd(\a_1,\a_2,\ldots, \a_r)=1$
or  $\gcd(\b_1,\b_2,\ldots, \b_t)=1$,
or $\gcd(\b_1,\b_2,\ldots, \b_t)=1$ or $\gcd(\gamma_\ell-\gamma_k, 0<\ell<k<n-1)=1$
	 there exist integer sequences $\{b_i\}_{1 \leqslant i \leqslant r}$, $\{c_j\}_{1 \leqslant j \leqslant t}$ and $\{d_{\ell,k}\}_{0\le \ell <k \le n-1}$ such that
	$$
	\sum_{i=1}^rb_i\a_i=1 ~~\text{or}~~\sum_{j=1}^tc_j\beta_j=1
	~~\text{or}~~\sum_{0\le \ell <k \le n-1}d_{\ell,k}(\gamma_\ell-\gamma_k)=1,
	$$
	which is equivalent to
	$$
	\sum_{i=1}^r2\pi h_ib_i=\theta ~~\text{or}~~ \sum_{j=1}^t2\pi n_jc_j=\theta
	~~\text{or}~~\sum_{0\le \ell <k \le n-1}2\pi m_{\ell,k}d_{\ell,k}=\theta .
	$$
	Therefore, $\cos\theta=1$ or $e^{\textbf{i}\theta}=1$. Since $\Gamma$ is connected, we have $s>0$.
	Then $P_j(e^{\textit{i}\theta})>0$ for $j=2,3,4$. 	This completes the proof.
\end{proof}

\begin{lemma}\label{two}
	Let $\Gamma$ be a connected bicirculant graph. Then $P_1(1)=P^{\prime}_1(1) = 0$ and $P^{\prime\prime}_1(1) < 0$.
\end{lemma}
\begin{proof}
	By simple computation, we obtain that $P_1(1)=P^{\prime}_1(1)=0$ and 
	$$
	\begin{aligned}
		P^{\prime\prime}_1(1)&=-2s\sum_{j=1}^{r}\a_j^2-2s\sum_{j=1}^{t}\b_j^2
		-2s\sum_{j=1}^{s}\gamma_j^2+2\left(\sum_{j=1}^{s}\gamma_j\right)^2.
	\end{aligned}
	$$
	By Lemma \ref{lag}, we have
	$$
	\begin{aligned}
		P^{\prime\prime}_1(1)&=-2\left(s\sum_{j=1}^{r}\a_j^2+s\sum_{j=1}^{t}\b_j^2
		+s\sum_{j=1}^{s}\gamma_j^2-\left(\sum_{j=1}^{s}\gamma_j\right)^2\right)\\
		&=-2\left(s\sum_{j=1}^{r}\a_j^2+s\sum_{j=1}^{t}\b_j^2+\sum_{1\le j<k\le s}(\gamma_j-\gamma_k)^2\right)\le 0.\\
	\end{aligned}
	$$
	Note that $P^{\prime\prime}_1(1)=0$ if and only if all elements in 
$\{\a_j\}_{j=1,2,\ldots,r}$ and $\{\b_j\}_{j=1,2,\ldots,t}$ and  $\{\gamma_j-\gamma_k\}_{j,k=1,2,\ldots,s}$ are equal to 0. This contradicts our assumption.

\end{proof}

\begin{prop}\cite{med}\label{gen}
Let $g(z)$ be a polynomial of degree $2k$ with integer coefficients. Suppose that all the roots of the polynomial $g(z)=0$ are $\xi_1, \xi_2, \ldots, \xi_{2k-1}, \xi_{2k}$. Then
$$
F(x)=\sum_{n=1}^{\infty} n \prod_{j=1}^{2k}\left(\xi_j^n-1\right) x^n
$$
is a rational function with integer coefficients.
Moreover, if $\xi_{j+k}=\xi_j^{-1}$ for $j=1,2, \ldots, k$, then $F(x)=F(1 / x)$.
\end{prop}

\section{Counting the number of spanning trees}
In this section, we aim to find new formulas for the number of spanning trees
of  the bicirculant graph in light of Chebyshev polynomials.

\begin{theorem}\label{th1}
Let $\Gamma=BC(\mathbb{Z}_n;R,T,S)$ be a bicirculant graph. Then we have
\begin{itemize}
\item[(1)] if $\Gamma=\Gamma_1$, then 
$$
\tau(\Gamma_1)=\frac{ns|a_k|^n}{q}
\prod_{\ell=1}^{k-1}|2T_n(w_\ell)-2|,
$$
\item[(2)] if $\Gamma=\Gamma_2$, then $$\tau(\Gamma_2)=
\frac{|b_k|^\frac{n}{2}|a_k|^\frac{n}{2}ns}{4q}\prod_{\ell=1}^{k}|2T_\frac{n}{2}(v_\ell)+2|\prod_{\ell=1}^{k-1}|2T_\frac{n}{2}(w_\ell)-2|,$$
\item[(3)] if $\Gamma=\Gamma_3$, then $$\tau(\Gamma_3)=
\frac{|c_k|^\frac{n}{2}|a_k|^\frac{n}{2}ns}{4q}\prod_{\ell=1}^{k}|2T_\frac{n}{2}(u_\ell)+2|\prod_{\ell=1}^{k-1}|2T_\frac{n}{2}(w_\ell)-2|,$$
\item[(4)] if $\Gamma=\Gamma_4$, then $$\tau(\Gamma_4)=
\frac{|d_k|^\frac{n}{2}|a_k|^\frac{n}{2}ns}{4q}\prod_{\ell=1}^{k}|2T_\frac{n}{2}(y_\ell)+2|\prod_{\ell=1}^{k-1}|2T_\frac{n}{2}(w_\ell)-2|,$$
\end{itemize}
where $q=s\sum_{j=1}^r\a_j^2+s\sum_{j=1}^t\b_j^2+\sum_{1\le j<i\le s}(\gamma_j-\gamma_i)^2$, and $a_k,b_k, c_k,d_k$ are the leading coefficients of $P_1(z),P_2(z),P_3(z),P_4(z)$, respectively, and $w_\ell,v_\ell,u_\ell,y_\ell$, $\ell=1,2,\ldots, k$, are all the roots of Chebyshev transform of $P_j(z)=0$ for $j=1,2,3,4$, respectively.
\end{theorem}
\begin{proof}
(1)  By Lemma \ref{adj}, we obtain that the adjacency matrix of the graph $\Gamma_1$ is given by the
$2n \times 2n$ block matrix
$$
A(\Gamma_1)=\begin{pmatrix}
\sum_{j=1}^r\left(Q^{\a_j}+Q^{-\a_j}\right)& \sum_{j=1}^s Q^{\gamma_j} \\
\sum_{j=1}^s Q^{-\gamma_j} & \sum_{j=1}^t\left(Q^{\beta_j}+Q^{-\beta_j}\right)
\end{pmatrix},
$$
where $Q=\operatorname{circ}(0,1,0, \ldots, 0)$. Let $L(\Gamma_1)$ be the Laplacian matrix of $\Gamma_1$. Then
 $$
L(\Gamma_1)=
\begin{pmatrix}
 k_1I_n-\sum_{j=1}^r\left(Q^{\a_j}+Q^{-\a_j}\right) & -\sum_{j=1}^s Q^{\gamma_j}\\
 -\sum_{j=1}^s Q^{-\gamma_j} & k_2I_n-\sum_{j=1}^t\left(Q^{\beta_j}+Q^{-\beta_j}\right)
 \end{pmatrix},
 $$
 where $k_1=2r+s$ and $k_2=2t+s$.
 Since the eigenvalues of circulant matrix $Q$ are $1, \V_n, \ldots,\V_n^{n-1}$,
there exists an invertible matrix $P$ such that $P^{-1}QP=\mathcal{K}_n={\rm diag}(1,\V_n,\ldots,\V_n^{n-1})$ with $\V_n=\exp(\frac{2\pi \mathbf{i}}{n})$.
Then we have
$$
\begin{pmatrix}
P^{-1} & 0 \\
0 & P^{-1}
\end{pmatrix}
L(\Gamma_1)
\begin{pmatrix}
P & 0 \\
0 & P
\end{pmatrix}=
\begin{pmatrix}
\mathcal{A}(\mathcal{K}_n) & \mathcal{C}(\mathcal{K}_n) \\
\mathcal{C}(\mathcal{K}_n^{-1}) & \mathcal{B}(\mathcal{K}_n)
\end{pmatrix}.$$
Hence $L(\Gamma_1)$ and $\mathcal{L}=
\begin{pmatrix}
\mathcal{A}(\mathcal{K}_n) & \mathcal{C}(\mathcal{K}_n) \\
\mathcal{C}(\mathcal{K}_n^{-1}) & \mathcal{B}(\mathcal{K}_n)
\end{pmatrix}$ have the same eigenvalues. Suppose $\lambda$ is an eigenvalue of $\mathcal{L}$. Then
$$
0=
\det\begin{pmatrix}
\mathcal{L}-\lambda I_{2n}
\end{pmatrix}
=
\det\begin{pmatrix}
\mathcal{A}(\mathcal{K}_n)-\lambda I_{n} & \mathcal{C}(\mathcal{K}_n) \\
\mathcal{C}(\mathcal{K}_n^{-1}) & \mathcal{B}(\mathcal{K}_n)-\lambda I_{n}
\end{pmatrix}.
$$
Since $(\mathcal{A}(\mathcal{K}_n)-\lambda I_{n})\mathcal{C}(\mathcal{K}_n^{-1})=\mathcal{C}(\mathcal{K}_n^{-1})(\mathcal{A}(\mathcal{K}_n)-\lambda I_{n})$,
by Lemma \ref{mat}, we have
$$
 \begin{aligned}
0&=
\det\begin{pmatrix}
(\mathcal{A}(\mathcal{K}_n)-\lambda I_{n})( \mathcal{B}(\mathcal{K}_n)-\lambda I_{n})-\mathcal{C}(\mathcal{K}_n^{-1})\mathcal{C}(\mathcal{K}_n)
\end{pmatrix}\\
&=
\det\begin{pmatrix}
\lambda^2 I_n -\lambda I_n(\mathcal{A}(\mathcal{K}_n)+\mathcal{B}(\mathcal{K}_n))+
\mathcal{A}(\mathcal{K}_n)\mathcal{B}(\mathcal{K}_n)-\mathcal{C}(\mathcal{K}_n^{-1})\mathcal{C}(\mathcal{K}_n)
\end{pmatrix}.
 \end{aligned}
$$
Hence, $\lambda$ is a root of the quadratic equation
$$
\lambda^2-\left(\mathcal{A}\left(\varepsilon_n^j\right)+\mathcal{B}\left(\varepsilon_n^{j}\right)\right)\lambda
+\mathcal{A}\left(\varepsilon_n^j\right) \mathcal{B}\left(\varepsilon_n^{j}\right)-\mathcal{C}\left(\varepsilon_n^{-j}\right) \mathcal{C}\left(\varepsilon_n^{j}\right)=0,
$$
where $j=0,1,2,\ldots,n-1$. The solutions of this equation are $\lambda_{j,1},\lambda_{j,2}$ for $j=0,1,2,\ldots,n-1$.
Then we have
$$
\lambda_{j,1}\lambda_{j,2}=P_1(\V_n^{j})=\mathcal{A}\left(\varepsilon_n^j\right) \mathcal{B}\left(\varepsilon_n^{j}\right)-\mathcal{C}\left(\varepsilon_n^{-j}\right) \mathcal{C}\left(\varepsilon_n^{j}\right).
$$
Since $\lambda_{0,1}\lambda_{0,2}=P_1(1)=\mathcal{A}(1)\mathcal{B}(1)-\mathcal{C}(1)\mathcal{C}(1)=0$, we have
$\lambda_{0,1}=0$ and $\lambda_{0,2}=2s$.
Therefore,
$$
\tau(\Gamma_1)=\frac{2s\prod_{j=1}^{n-1}\lambda_{j,1}\lambda_{j,2}}{2n}=\frac{s}{n}\prod_{j=1}^{n-1} P_1\left(\varepsilon_n^j\right).
$$
Note that $P_1(z)=P_1(z^{-1})$ and 
$1$ is a root of $P_1(z)$ with multiplicity 2. By Lemmas \ref{che} and \ref{two}, we have
$$
\prod_{j=1}^{n-1} P_1\left(\varepsilon_{n}^{j}\right)
 =\frac{a_k}{-2q_1}(-1)^{nk-1}a_k^{n-1}n^2 2^k\prod_{\ell=1}^{k-1}
(T_n(w_\ell)-1),
$$
where $s\sum_{j=1}^{r}\a_j^2+s\sum_{j=1}^{t}\b_j^2+\sum_{1\le j<k\le s}(\gamma_j-\gamma_k)^2$.
Hence
$$
\tau(\Gamma_1)=\frac{(-1)^{nk}na_k^ns}{q_1}
\prod_{\ell=1}^{k-1}(2T_n(w_\ell)-2).
$$
Since $\tau(\Gamma_1)$ is a positive integer, the result follows.

\medskip

(2) Let $\lambda$ be the eigenvalue of $L(\Gamma_2)$. Similarly as in the proof of $(1)$,
$\lambda$ is a root of the quadratic equation
$$
\lambda^2-\left(\mathcal{A}\left(\varepsilon_n^j\right)+1-\V_n^{\frac{jn}{2}}+\mathcal{B}\left(\varepsilon_n^{j}\right)\right)\lambda
+\left(\mathcal{A}\left(\varepsilon_n^j\right)+1-\V_n^{\frac{jn}{2}}\right) \mathcal{B}\left(\varepsilon_n^{j}\right)-\mathcal{C}\left(\varepsilon_n^{-j}\right) \mathcal{C}\left(\varepsilon_n^{j}\right)=0,
$$
where $j=0,1,2,\ldots,n-1$. The solutions of this equation are $\lambda_{j,1},\lambda_{j,2}$ for $j=0,1,2,\ldots,n-1$.
Then we have
$$
\lambda_{j,1}\lambda_{j,2}=\left(\mathcal{A}\left(\varepsilon_n^j\right)+1-\V_n^{\frac{jn}{2}}\right) \mathcal{B}\left(\varepsilon_n^{j}\right)-\mathcal{C}\left(\varepsilon_n^{-j}\right) \mathcal{C}\left(\varepsilon_n^{j}\right).
$$
Note that $\lambda_{j,1}\lambda_{j,2}=P_1(\V_n^j)=\mathcal{A}(\varepsilon_n^j) \mathcal{B}(\varepsilon_n^{j})-\mathcal{C}(\varepsilon_n^{-j}) \mathcal{C}(\varepsilon_n^{j})$
for even $j$, and
$\lambda_{j,1}\lambda_{j,2}=P_2(\V_n^j)=\left(\mathcal{A}(\varepsilon_n^j)+2\right) \mathcal{B}(\varepsilon_n^{j})-\mathcal{C}(\varepsilon_n^{-j}) \mathcal{C}(\varepsilon_n^{j})$
for odd $j$. Since $\lambda_{0,1}\lambda_{0,2}=P_1(1)=\mathcal{A}(1)\mathcal{B}(1)-\mathcal{C}(1)\mathcal{C}(1)=0$, we have
$\lambda_{0,1}=0$ and $\lambda_{0,2}=2s$.
Therefore,
$$
\begin{aligned}
\tau(\Gamma_2)&=\frac{s}{n}\prod_{j=1}^{n-1}\lambda_{j,1}\lambda_{j,2}
=\frac{s}{n}\prod_{j=0}^{\frac{n}{2}-1}P_{2}(\V_n^{2j+1})\prod_{j=1}^{\frac{n}{2}-1}P_{1}(\V_n^{2j})\\
&=\frac{s}{n}\frac{\prod_{j=0}^{n-1}P_{2}(\V_n^j)}{\prod_{j=0}^{\frac{n}{2}-1}P_{2}(\V_n^{2j})}\prod_{j=1}^{\frac{n}{2}-1}P_{1}(\V_n^{2j})\\
&=\frac{s}{n}\frac{\prod_{j=0}^{n-1}P_{2}(\V_n^j)}{\prod_{j=0}^{\frac{n}{2}-1}P_{2}(\V_{\frac{n}{2}}^{j})}\prod_{j=1}^{\frac{n}{2}-1}P_{1}(\V_\frac{n}{2}^{j})\\
&=\frac{s}{n}\frac{\prod_{j=1}^{n-1}P_{2}(\V_n^j)}{\prod_{j=1}^{\frac{n}{2}-1}P_{2}(\V_{\frac{n}{2}}^{j})}\prod_{j=1}^{\frac{n}{2}-1}P_{1}(\V_\frac{n}{2}^{j}),\\
\end{aligned}
$$
where $P_1(z)$ and $P_2(z)$ are defined in $(3)$ and $(4)$, respectively. By Lemma \ref{one}, 
$P_2(1)\not=0$. 
Based on the discussion in (1) and Lemma \ref{ch}, we have
\begin{itemize}
\item[(i)] $\prod_{j=1}^{n-1}P_2(\V_n^j)=(-1)^{nk-k}b_k^{n-1}
\prod_{\ell=1}^{k}\frac{T_n(v_\ell)-1}{v_\ell-1}$,
\item[(ii)] $\prod_{j=1}^{\frac{n}{2}-1}P_2(\V_{\frac{n}{2}}^j)=(-1)^{\frac{nk}{2}-k}b_k^{\frac{n}{2}-1}
\prod_{\ell=1}^{k}\frac{T_\frac{n}{2}(v_\ell)-1}{v_\ell-1}$,
\item[(iii)] $\prod_{j=1}^{\frac{n}{2}-1}P_1(\V_\frac{n}{2}^j)=\frac{(-1)^{\frac{nk}{2}}a_k^\frac{n}{2}n^2}{4q_1}\prod_{\ell=1}^{k-1}(2T_\frac{n}{2}(w_\ell)-2)$,
\end{itemize}
where $b_k$ and $a_k$ are the leading coefficients of $P_2(z)$ and $P_1(z)$, respectively, and $v_\ell$ and $w_\ell$, $\ell=1,2,\ldots,k$,
are all the roots of Chebyshev transform of $P_2(z)=0$ and $P_1(z)=0$, respectively.
Since $T_n\left(v_\ell\right)-1=2(T_\frac{n}{2}\left(v_\ell\right)-1)(T_\frac{n}{2}\left(v_\ell\right)+1)$, we have
$$
\begin{aligned}
\tau(\Gamma_2)&=\frac{s}{n}\frac{(-1)^{nk-k}b_k^{n-1}\prod_{\ell=1}^{k}(T_n(v_{\ell})-1)}
{(-1)^{\frac{nk}{2}-k}b_k^{\frac{n}{2}-1}\prod_{\ell=1}^{k}(T_\frac{n}{2}(v_\ell)-1)}
\frac{(-1)^{\frac{nk}{2}}a_k^\frac{n}{2}n^2}{4q_1}\prod_{\ell=1}^{k-1}(2T_\frac{n}{2}(w_\ell)-2)\\
&=\frac{(-1)^{nk}b_k^\frac{n}{2}a_k^\frac{n}{2}ns}{4q_1}\prod_{\ell=1}^{k}(2T_\frac{n}{2}(v_\ell)+2)\prod_{\ell=1}^{k-1}(2T_\frac{n}{2}(w_\ell)-2).
\end{aligned}
$$
Since $\tau(\Gamma_2)$ is a positive integer, the result follows.

\medskip

Similarly, we can prove that  both $(3)$ and $(4)$ hold.
\end{proof}

\section{Arithmetic properties of the number of spanning trees}
Recall that any positive integer $u$ can be uniquely represented in the form $u=vr^2$,
where $u$ and $v$ are positive integers and $v$ is square-free. We will call $v$ the \textit{square-free part} of $u$.
The main result of this section is the following theorem.

\begin{theorem}\label{th2}
Let $\tau_{\Gamma}(2n)$ be the number of  spanning trees of bicirculant graph $\Gamma$.
Denote by $k_1$ (resp. $k_2$) the number of odd
numbers (resp. even numbers) in $\{\a_1, \a_2, \ldots, \a_r\}$. Denote by $m_1$ (resp. $m_2$) the
number of odd numbers (resp. even numbers) in  $\{\b_1, \b_2, \ldots, \b_t\}$.
Denote by $h_1$ (resp. $h_2$) the
number of odd numbers (resp. even numbers) in  $\{\gamma_1, \gamma_2, \ldots, \gamma_s\}$.
Then we have
\begin{itemize}
\item[(1)] if $\Gamma=\Gamma_1$, then there exist two integer sequences $a_1(n)$ and $b_1(n)$ such that
\begin{eqnarray}
\tau(\Gamma_1)=\begin{cases}
nsa_1(n)^2, &n~is~odd,\\
nsq_1 b_1(n)^2, &n~is~even,\\
\end{cases}
\nonumber
\end{eqnarray}

\item[(2)] if $\Gamma=\Gamma_2$, then there exist two integer sequences $a_2(n)$ and $b_2(n)$ such that
\begin{eqnarray}
\tau(\Gamma_2)=\begin{cases}
\frac{1}{4}nsq_2 a_2(n)^2, &\frac{n}{2}~is~odd,\\
\frac{1}{4}nsq_1 b_2(n)^2, &\frac{n}{2}~is~even,\\
\end{cases}
\nonumber
\end{eqnarray}

\item[(3)] if $\Gamma=\Gamma_3$, then there exist two integer sequences $a_3(n)$ and $b_3(n)$ such that
\begin{eqnarray}
\tau(\Gamma_3)=\begin{cases}
\frac{1}{4}nsq_3 a_3(n)^2, &\frac{n}{2}~is~odd,\\
\frac{1}{4}nsq_1b_3(n)^2, &\frac{n}{2}~is~even,\\
\end{cases}
\nonumber
\end{eqnarray}

\item[(4)] if $\Gamma=\Gamma_4$, then there exist two integer sequences $a_4(n)$ and $b_4(n)$ such that
\begin{eqnarray}
\tau(\Gamma_4)=\begin{cases}
\frac{1}{4}nsq_4 a_4(n)^2, &\frac{n}{2}~is~odd,\\
\frac{1}{4}nsq_1 b_4(n)^2, &\frac{n}{2}~is~even,\\
\end{cases}
\nonumber
\end{eqnarray}
\end{itemize}
where $q_1,q_2,q_3,q_4$ are the square-free parts of $(4k_1+s)(4m_1+s)-(h_2-h_1)^2,
(4k_1+s+2)(4m_1+s)-(h_2-h_1)^2,(4k_1+s)(4m_1+s+2)-(h_2-h_1)^2,(4k_1+s+2)(4m_1+s+2)-(h_2-h_1)^2$, respectively.
\end{theorem}

\begin{proof}
From the proof of Theorem \ref{th1}(1), we have
$$
\tau(\Gamma_1)=\frac{s}{n}\sum_{j=1}^n\lambda_{j,1}\lambda_{j,2}=\frac{s}{n}\prod_{j=1}^{n-1} P_1\left(\varepsilon_n^j\right)
$$
and $1$ is the root of $P_1(z)=0$ with  multiplicity 2. Since $\lambda_{j,1}\lambda_{j,2}=P_1(\V_n^j)=\lambda_{n-j,1}\lambda_{n-j,2}$,
we have $\tau(\Gamma_1)=\frac{s}{n}\left( \sum_{j=1}^{\frac{n-1}{2}}\lambda_{j,1}\lambda_{j,2}\right)^2$ if $n$ is odd,
and $\tau(\Gamma_1)=\frac{s}{n}\lambda_{\frac{n}{2},1}\lambda_{\frac{n}{2},2}
\left( \sum_{j=1}^{\frac{n}{2}-1}\lambda_{j,1}\lambda_{j,2}\right)^2$ if $n$ is even.
Then we have
$$
\begin{aligned}
\lambda_{\frac{n}{2},1}\lambda_{\frac{n}{2},2}&=P_1(-1)=\mathcal{A}(1)\mathcal{B}(1)-\mathcal{C}(1)\mathcal{C}(1)\\
&=\left(2r+s-(2k_2-2k_1)\right)\left(2t+s-(2m_2-2m_1)\right)-(h_2-h_1)^2\\
&=(4k_1+s)(4m_1+s)-(h_2-h_1)^2.
\end{aligned}
$$ Since each algebraic number $\lambda_{i,j}$ comes into both products $\prod_{j=1}^{\frac{n}{2}-1}\lambda_{j,1}\lambda_{j,2}$ and  $\prod_{j=1}^{\frac{n-1}{2}}\lambda_{j,1}\lambda_{j,2}$ with all of its Galois conjugate elements \cite{DL}, we have $\prod_{j=1}^{\frac{n}{2}-1}\lambda_{j,1}\lambda_{j,2}$ and $\prod_{j=1}^{\frac{n-1}{2}}\lambda_{j,1}\lambda_{j,2}$ are integers. Since $1$ is the root of $P_1(z)=0$ with  multiplicity 2, we have $P_1(z)=\frac{a_k}{z^k}(z-1)^2C_1(z)$, where $C_1(z)=\prod_{j=1}^{k-1}(z-z_j)(z-z_j^{-1})$, and
$z_\ell,z_\ell^{-1}$, $\ell=1,2,\ldots,k-1$, are all the roots different from $1$ of $P_1(z)=0$.
Hence
$$
\begin{aligned}
\frac{\tau(\Gamma_1)}{ns}&=\frac{1}{n^2}\prod_{j=1}^{n-1} P_1\left(\varepsilon_n^j\right)=\frac{(-1)^{nk-k}a_k^{n-1}}{n^2}\prod_{j=1}^{n-1}(\V_n^j-1)^2 \prod_{j=1}^{n-1}C_1(\V_n^j)\\
&=\frac{(-1)^{nk-k}a_k^{n-1}}{n^2}\left(\lim_{z\to1}\frac{z^n-1}{z-1}\right)^2
\prod_{\ell=1}^{k-1}\frac{z_\ell^n-1}{z_\ell-1}\\
&=(-1)^{nk-k}a_k^{n-1}\prod_{\ell=1}^{k-1}\frac{z_\ell^n-1}{z_\ell-1}\\
&=(-1)^{nk-k}\mathrm{Res}(a_kC_1(z), \frac{z^n-1}{z-1}).
\end{aligned}
$$
Note that $a_kC_1(z)$ and $\frac{z^n-1}{z-1}$ are polynomials with integer coefficients. Then $\mathrm{Res}(a_kC_1(z), \frac{z^n-1}{z-1})$
is an integer and hence $\frac{\tau(\Gamma_1)}{ns}$ is an integer.
Let $q_1$ be the square-free part of $(4k_1+s)(4m_1+s)-(h_2-h_1)^2$. Then $(4k_1+s)(4m_1+s)-(h_2-h_1)^2=q_1u_1^2$, where $u_1$ is
an integer. Therefore,
\begin{eqnarray}
\frac{\tau(\Gamma_1)}{ns}=\begin{cases}
\left( \frac{\sum_{j=1}^{\frac{n-1}{2}}\lambda_{j,1}\lambda_{j,2}}{n}\right)^2, &n~is~odd,\\
q_1\left(\frac{u_1\sum_{j=1}^{\frac{n}{2}-1}\lambda_{j,1}\lambda_{j,2}}{n}\right)^2, &n~is~even.\\
\end{cases}
\nonumber
\end{eqnarray}
Setting $a_1(n)=\frac{\sum_{j=1}^{\frac{n-1}{2}}\lambda_{j,1}\lambda_{j,2}}{n}$. Then $a_1(n)=\frac{b}{a}$ is a rational number,
where $a,b$ are integers and $\gcd(a,b)=1$. Since $a_1(n)^2=\frac{b^2}{a^2}$ is an integer, $a^2$ divides $b^2$. Since
$\gcd(a,b)=1$, we obtain $a=1$. Hence $a_1(n)$ is an integer. Setting $b_1(n)=\frac{u_1\sum_{j=1}^{\frac{n}{2}-1}\lambda_{j,1}\lambda_{j,2}}{n}$.
By the same argument, we have that $b_1(n)$ is an integer.

\medskip

(2) From the proof of Theorem \ref{th1}(2), we have
$$
\tau(\Gamma_2)=\frac{s}{n}\sum_{j=1}^n\lambda_{j,1}\lambda_{j,2}
=\frac{s}{n}\prod_{j=0}^{\frac{n}{2}-1}P_{2}(\V_n^{2j+1})\prod_{j=1}^{\frac{n}{2}-1}P_{1}(\V_\frac{n}{2}^{j}).
$$
Note that $\lambda_{j,1}\lambda_{j,2}=P_1(\V_n^j)=P_1(\V_n^{n-j})=\lambda_{n-j,1}\lambda_{n-j,2}$ if $j$ is even, and
$\lambda_{j,1}\lambda_{j,2}=P_2(\V_n^j)=P_2(\V_n^{n-j})=\lambda_{n-j,1}\lambda_{n-j,2}$ if $j$ is odd.
Then $\lambda_{j,1}\lambda_{j,2}=\lambda_{n-j,1}\lambda_{n-j,2}$ for $j=1,2,\ldots,n-1$.
Hence $\tau(\Gamma_2)=\frac{s}{n}\lambda_{\frac{n}{2},1}\lambda_{\frac{n}{2},2}
\left( \sum_{j=1}^{\frac{n}{2}-1}\lambda_{j,1}\lambda_{j,2}\right)^2$, where
$\lambda_{\frac{n}{2},1}\lambda_{\frac{n}{2},2}=(4k_1+s+2)(4m_1+s)-(h_2-h_1)^2$ for odd $\frac{n}{2}$, and
$\lambda_{\frac{n}{2},1}\lambda_{\frac{n}{2},2}=(4k_1+s)(4m_1+s)-(h_2-h_1)^2$ for even $\frac{n}{2}$.
Since each algebraic number $\lambda_{i,j}$ comes into the product
 $\prod_{j=1}^{\frac{n}{2}-1}\lambda_{j,1}\lambda_{j,2}$ with all of its Galois conjugate elements \cite{DL},
 we have $\prod_{j=1}^{\frac{n}{2}-1}\lambda_{j,1}\lambda_{j,2}$ is an integer.
 Note that $P_2(1)\not=0$. By Lemma \ref{book}, we obtain 
 $$
 \begin{aligned}\prod_{j=0}^{\frac{n}{2}-1}P_{2}(\V_n^{2j+1})&=(-1)^{\frac{nk}{2}}\mathrm{Res}(z^\frac{n}{2}+1,z^kP_2(z)),\\
 \prod_{j=1}^{\frac{n}{2}-1}P_{1}(\V_\frac{n}{2}^{j})&=(-1)^{\frac{nk}{2}-k}\frac{n^2}{4}\mathrm{Res}(\frac{z^\frac{n}{2}-1}{z-1},\frac{z^kP_1(z)}{(z-1)^2}).
 \end{aligned}
 $$
 Then $\frac{4\tau(\Gamma_2)}{ns}=(-1)^{nk-k}\mathrm{Res}(z^\frac{n}{2}+1,z^kP_2(z))
 \mathrm{Res}(\frac{z^\frac{n}{2}-1}{z-1},\frac{z^kP_1(z)}{(z-1)^2})$. Since $z^\frac{n}{2}+1, z^kP_2(z),
 \frac{z^\frac{n}{2}-1}{z-1}$ and $\frac{z^kP_1(z)}{(z-1)^2}$
 are polynomials with integer coefficients, $\mathrm{Res}(z^kP_2(z),z^\frac{n}{2}+1)$ and $
 \mathrm{Res}(z^kP_1(z),\frac{z^\frac{n}{2}-1}{z-1})$ are integers and hence $\frac{4\tau(\Gamma_2)}{ns}$ is an integer.
 Let $q_2$ and $q_1$ be the square-free parts of $(4k_1+s+2)(4m_1+s)-(h_2-h_1)^2$ and $(4k_1+s)(4m_1+s)-(h_2-h_1)^2$, respectively. Then $(4k_1+s+2)(4m_1+s)-(h_2-h_1)^2=q_2u_2^2$ and $(4k_1+s)(4m_1+s)-(h_2-h_1)^2)=q_1 v_2^2$ for some integers $u_1,v_1$. Therefore,
 \begin{eqnarray}
\frac{4\tau(\Gamma_1)}{ns}=\begin{cases}
q_2\left( \frac{u_2\sum_{j=1}^{\frac{n}{2}-1}\lambda_{j,1}\lambda_{j,2}}{n}\right)^2, &\frac{n}{2}~is~odd,\\
q_1\left(\frac{u_1\sum_{j=1}^{\frac{n}{2}-1}\lambda_{j,1}\lambda_{j,2}}{n}\right)^2, &\frac{n}{2}~is~even.\\
\end{cases}
\nonumber
\end{eqnarray}
Setting $a_2(n)=\frac{u_2\sum_{j=1}^{\frac{n}{2}-1}\lambda_{j,1}\lambda_{j,2}}{n}$ and $b_2(n)=\frac{u_1\sum_{j=1}^{\frac{n}{2}-1}\lambda_{j,1}\lambda_{j,2}}{n}$.
By the same argument as $(1)$, we have that $a_2(n)$ and $b_2(n)$ are integers.

\medskip

Similarly, we can prove that  both $(3)$ and $(4)$ hold.
\end{proof}

\begin{remark}\label{re41}
For cases $(2),(3)$ and $(4)$. If one of $\frac{n}{2}, s, q_j$ is even, then $\frac{nsq_j}{4}$ is an integer for $j=2,3,4$.
Suppose that $\frac{n}{2}, s, q_j$ are odd. Since $\tau(\Gamma_2),\tau(\Gamma_3), \tau(\Gamma_4)$ are integers,
 we have that $a_2(n),a_3(n)$ and $a_4(n)$ are even. 
\end{remark}

\section{Asymptotic for the number of spanning trees}
 In this section, we give asymptotic formulas for the number of  spanning trees of
bicirculant graphs. Two functions $f(n)$ and $g(n)$ are said to be \textit{asymptotically equivalent}, as $n\rightarrow \infty$ if $\text{lim}_{n\rightarrow \infty}\frac{f (n)}{g(n)}=1$. We will write $f (n)\sim g(n)$, $n \rightarrow \infty$ in this case. We suppose that parameters
$\a_1,\a_2,\ldots,\a_r$, $\beta_1,\b_2,\ldots,\beta_t$, $\gamma_1,\gamma_2,\ldots,\gamma_s$ are fixed,
and the inequalities
$$
0 < \a_1 < \a_2 < \dots < \a_r < n,\quad
0 < \beta_1 < \beta_2 < \dots < \beta_t < n,
\quad
0 \le \gamma_1 < \gamma_2 < \dots < \gamma_t \le n-1
$$
hold for all sufficiently large values of $n$.
According to  conditions $(a),(b),(c)$,  we have $(n,d_1)=1$ or $(n,d_2)=1$ or $(n,d_3)=1$, where
$
d_1=\mathrm{gcd}(\a_1,\a_2,\ldots,\a_r),d_2=\mathrm{gcd}(\b_1,\b_2,\ldots,\b_t),
d_3=\gcd(\gamma_\ell-\gamma_k, 0\le\ell<k\le n-1).
$
Since $n$ is an arbitrarily large number and $d_1,d_2,d_3$ is fixed,
we can choose $n$ such that it is a multiple of $d_1$, $d_2$ or $d_3$.
Then $d_i=(n,d_i)=1$ for $i=1,2,3$, so that the condition of Lemma \ref{one} is satisfied.

\begin{theorem}\label{th3}
Let $\Gamma=BC(\mathbb{Z}_n;R,T,S)$ be a bicirculant graph. Then we have
\begin{itemize}
\item[(1)] if $\Gamma=\Gamma_1$, then
$$
\tau(\Gamma_1)\sim
\frac{ns}{q}A^{n},~~n\rightarrow \infty,
$$
where $q=s\sum_{j=1}^r\a_j^2+s\sum_{j=1}^t\b_j^2+\sum_{1\le j<i\le s}(\gamma_j-\gamma_i)^2$ and $A=\exp \left(\int_0^1 \log \left|P_1\left(e^{2 \pi \mathbf{i} t}\right)\right| d t\right)$ is the Mahler measure of the Laurent polynomial $P_1(z)$,

\item[(2)] if $\Gamma=\Gamma_2$, then
$$
\tau(\Gamma_2)\sim \frac{ns}{4q}B^{n},~~n\rightarrow \infty,
$$
where $q=s\sum_{j=1}^r\a_j^2+s\sum_{j=1}^t\b_j^2+\sum_{1\le j<i\le s}(\gamma_j-\gamma_i)^2$ and
$B=\exp \left(\int_0^1 \log \left|P_2\left(e^{2 \pi \mathbf{i} t}\right)P_1\left(e^{2 \pi \mathbf{i} t}\right)\right| d t\right)$ is the Mahler measure of the Laurent polynomial $P_2(z)P_1(z)$,

\item[(3)] if $\Gamma=\Gamma_3$, then
$$
\tau(\Gamma_3)\sim \frac{ns}{4q}C^{n},~~n\rightarrow \infty,
$$
where $q=s\sum_{j=1}^r\a_j^2+s\sum_{j=1}^t\b_j^2+\sum_{1\le j<i\le s}(\gamma_j-\gamma_i)^2$  and
$C=\exp \left(\int_0^1 \log \left|P_3\left(e^{2 \pi \mathbf{i} t}\right)P_1\left(e^{2 \pi \mathbf{i} t}\right)\right| d t\right)$ is the Mahler measure of the Laurent polynomial $P_3(z)P_1(z)$,

\item[(4)] if $\Gamma=\Gamma_4$, then
$$
\tau(\Gamma_4)\sim \frac{ns}{4q}D^{n},~~n\rightarrow \infty,
$$
where $q=s\sum_{j=1}^r\a_j^2+s\sum_{j=1}^t\b_j^2+\sum_{1\le j<i\le s}(\gamma_j-\gamma_i)^2$ and
$D=\exp \left(\int_0^1 \log \left|P_4\left(e^{2 \pi \mathbf{i} t}\right)P_1\left(e^{2 \pi \mathbf{i} t}\right)\right| d t\right)$ is the Mahler measure of the Laurent polynomial $P_4(z)P_1(z)$.
\end{itemize}
\end{theorem}
\begin{proof}
By Theorem \ref{th1}(1), the number of spanning trees $\tau(\Gamma_1)$ is given by
$$
\tau(\Gamma_1)=\frac{|a_k|^n ns}{q}\prod_{\ell=1}^{k-1}|2T_n(w_\ell)-2|,
$$
where $w_\ell=\frac{1}{2}(z_\ell+z_\ell^{-1})$ and $z_\ell,z_\ell^{-1}$, $\ell=1,2,\ldots,k-1$, are all the roots different from 1 of $P_1(z)=0$.
 By lemma \ref{one}, $z_\ell$ and $z_\ell^{-1}$ have the property 
$|z_\ell|\not=1$, $j=1,2,\ldots,k-1$. Replacing $z_\ell$ by $z_\ell^{-1}$ if necessary, we can assume that
$|z_{\ell}|>1$ for all $\ell = 1, 2,\ldots, k-c-1$.
Since $T_n(\frac{1}{2}(z_{\ell}+z_{\ell}^{-1}))=\frac{1}{2}(z_{\ell}^n+z_{\ell}^{-n})$, we have
$$T_n(w_\ell)\sim \frac{1}{2}z_\ell^n, \quad |2T_n(w_\ell)-2|\sim |z_{\ell}|^n, ~~n\rightarrow \infty.$$
Then we have
$$
|a_k|^n\prod_{\ell=1}^{k-1}\left|2T_n\left(w_\ell\right)-2\right|\sim |a_k|^n\prod_{\ell=1}^{k-1}|z_{\ell}|^n
=|a_k|^n\prod_{\substack{P_1(z)=0,\\|z|>1}}|z|^n=A^{n},~~n\rightarrow \infty,
$$
where $A=|a_k|\prod_{P_1(z)=0,|z|>1}|z| $ is the Mahler measure of $P_1(z)$. By (1) and (2), we have
$A=\exp \left(\int_0^1 \log \left|P_1\left(e^{2 \pi \mathbf{i} t}\right)\right| d t\right)$.

Finally,
$$\tau(\Gamma_1)=\frac{|a_k|^nns}{q}\prod_{\ell=1}^{k-1}|2T_n(w_\ell)-2|\sim \frac{ ns}{q}A^n, ~~n\rightarrow \infty.$$

\medskip

(2) By Theorem \ref{th1}(2), we obtain
$$
\tau(\Gamma_2)=\frac{|b_k|^\frac{n}{2}|a_k|^\frac{n}{2}ns}{4q}
\prod_{\ell=1}^{k}|2T_\frac{n}{2}(v_\ell)+2|\prod_{\ell=1}^{k-1}|2T_\frac{n}{2}(w_\ell)-2|,
$$
where $v_1,v_2,\ldots,v_k$ are all the roots of Chebyshev transform of $P_2(z)=0$ and $w_1,w_2,\ldots,w_{k-1}$ are all the roots different from 1 of Chebyshev transform of $P_1(z)=0$.  By lemma \ref{one}, $P_2(e^{\textit{i}\theta})\not=0$.
 Then $|z_\ell|\not=1$, where $z_\ell,z_\ell^{-1}$, $\ell=1,2,\ldots,k$, are all the roots of $P_2(z)=0$.
By the similar argument as $(1)$, we obtain
$$
|b_k|^\frac{n}{2}\prod_{\ell=1}^k\left|2T_\frac{n}{2}\left(v_\ell\right)+2\right|\sim
|b_k|^\frac{n}{2}\prod_{\substack{P_2(z)=0,\\|z|>1}}|z|^n=B_2^{\frac{n}{2}},~~n\rightarrow \infty
$$
 and
 $$
|a_k|^\frac{n}{2}\prod_{\ell=1}^{k-1}\left|2T_\frac{n}{2}\left(w_\ell\right)-2\right|=
|a_k|^\frac{n}{2}\prod_{\substack{P_1(z)=0,\\|z|>1}}|z|^\frac{n}{2}=B_1^{\frac{n}{2}},~~n\rightarrow \infty,
$$
where $B_2=|b_k|\prod_{P_2(z)=0,|z|>1}|z| $ and $B_1=|a_k|\prod_{P_1(z)=0,|z|>1}|z|$ are the Mahler measure of $P_2(z)$ and $P_1(z)$, respectively.
Therefore,
$$
\tau(\Gamma_2)\sim
\frac{ns}{4q}B^{n},~~n\rightarrow \infty,
$$
where $B=B_1B_2=\exp \left(\int_0^1 \log \left|P_2\left(e^{2 \pi \mathbf{i} t}\right)P_1\left(e^{2 \pi \mathbf{i} t}\right)\right| d t\right)$ is the Mahler measure of the Laurent polynomial $P_2(z)P_1(z)$.

\medskip

Similarly, we can prove that  both $(3)$ and $(4)$ hold.
\end{proof}

\section{Generating function for the number of spanning trees}
In this section, our aim is establish the rationality of generating function for the number
of spanning trees of bicirculant graphs. Note that $n$ must be even for $\Gamma_2,\Gamma_3,\Gamma_4$.
Hence we can replace $n$ with $2n$ in these cases.

\begin{theorem}\label{th4}
Let $\Gamma_j=BC(\mathbb{Z}_n;R_j,T_j,S_j)$ be the bicirculant graph for $j=1,2,3,4$. Then we have
$$
\begin{aligned}
F_1(x)&=\sum_{n=1}^{\infty}\tau(\Gamma_1)x^n,~F_2(x)=\sum_{n=1}^{\infty}\tau_{\Gamma_2}(4n)x^n,~
F_3(x)=\sum_{n=1}^{\infty}\tau_{\Gamma_3}(4n)x^n,~F_4(x)=\sum_{n=1}^{\infty}\tau_{\Gamma_4}(4n)x^n
\end{aligned}
$$
are rational functions with integer coefficients.
Moreover, $F_j(\ell_jx)=\frac{1}{F_j(\ell_jx)}$,
where $\ell_1=a_k$, $\ell_2=a_kb_k, \ell_3=a_kc_k, \ell_4=a_kd_k$ and $a_k, b_k, c_k, d_k$
are the leading coefficients of $P_1(z)$, $P_2(z)$, $P_3(z)$, $P_4(z)$, respectively.
\end{theorem}
\begin{proof}
(1) From the proof Theorem \ref{th1}(1), we obtain
$$
\tau(\Gamma_1)=\frac{s}{n}\prod_{j=1}^{n-1} P_1\left(\varepsilon_n^j\right)
=\frac{(-1)^{nk}a_k^n ns}{q}\prod_{\ell=1}^{k-1}(2T_n(w_\ell)-2),
$$
where $q=s\sum_{j=1}^r\a_j^2+s\sum_{j=1}^t\b_j^2+\sum_{1\le j<i\le s}(\gamma_j-\gamma_i)^2$ and $a_k$ is the leading coefficient of $P_1(z)$
 and $w_1, w_2,\ldots, w_k$ are all the roots of Chebyshev transform of $P_1(z)=0$. Hence
 $$
\tau(\Gamma_1)=\frac{(-1)^{nk}a_k^n ns}{q}\prod_{\ell=1}^{k-1}(2T_n(w_\ell)-2)
=\frac{(-1)^{nk+k-1}a_k^n ns}{q}\prod_{\ell=1}^{k-1}(z_\ell^n-1)(z_\ell^{-n}-1),
$$
where $z_\ell,z_\ell^{-1}$, $\ell=1,2,\ldots,k-1$, are all the roots different from 1 of $P_1(z)=0$.
Hence
$$
F_1(x)=\tau(\Gamma_1)x^n=\frac{(-1)^{k-1}s}{q}\sum_{n=1}^{\infty}\left(n\prod_{\ell=1}^{k-1}
(z_\ell^n-1)(z_\ell^{-n}-1)\right)((-1)^ka_kx)^n.
$$
Since $q$ and $s$ are integers, by Proposition \ref{gen}, we have $F_1(x)$ is a rational function with integer
coefficients satisfying $F_1((-1)^ka_kx)=\frac{1}{F_1((-1)^ka_kx)}$. Replacing
$x$ with $(-1)^kx$, we have $F_1(a_kx)=\frac{1}{F_1(a_kx)}$.
\medskip

(2) From the proof Theorem \ref{th1}(2), we obtain
$$
\tau(\Gamma_2)=\frac{(-1)^{nk}b_k^\frac{n}{2}a_k^\frac{n}{2}ns}{4q}\prod_{\ell=1}^{k}
(2T_\frac{n}{2}(v_\ell)+2)\prod_{\ell=1}^{k-1}(2T_\frac{n}{2}(w_\ell)-2).
$$
Since $n$ is even, we can replace $n$ with $2n$. Hence we have
$$
\begin{aligned}
\tau_{\Gamma_2}(4n)&=\frac{b_k^n a_k^nns}{2q}\prod_{\ell=1}^{k}
(2T_n(v_\ell)+2)\prod_{\ell=1}^{k-1}(2T_n(w_\ell)-2)\\
&=\frac{(-1)^{k-1}b_k^n a_k^nns}{2q}\prod_{\ell=1}^{k}(z_\ell^n+1)(z_\ell^{-n}+1)\prod_{\ell=1}^{k-1}(\zeta_\ell^n-1)(\zeta_\ell^{-n}-1),
\end{aligned}
$$
where $z_\ell,z_\ell^{-1}$, $\ell=1,2,\ldots,k$, are all the roots of $P_2(z)=0$, and
$\zeta_\ell,\zeta_\ell^{-1}$, $\ell=1,2,\ldots,k-1$, are all the roots different from 1 of $P_1(z)=0$.
$$
F_2(x)=\tau(\Gamma_2)x^n=\frac{(-1)^{k-1} s}{2q}\left(n\prod_{\ell=1}^{k}(z_\ell^n+1)(z_\ell^{-n}+1)\prod_{\ell=1}^{k-1}(\zeta_\ell^n-1)(\zeta_\ell^{-n}-1)\right)(a_kb_kx)^n.
$$
Since $q$ and $s$ are integers, by Proposition \ref{gen}, we have $F_2(x)$ is a rational function with integer
coefficients satisfying $F_2(a_kb_kx)=\frac{1}{F_2(a_kb_kx)}$.

\medskip
Similarly, we can prove that  both $(3)$ and $(4)$ hold.
\end{proof}

\section{Examples}
In this section, we give some examples to illustrate our results.

\medskip

\textit{(1)} \underline{The graph $\Gamma=BC(\mathbb{Z}_n;\{1,-1\},\{1,-1\},\{0\})$}. In this case, we have $\Gamma=\Gamma_1$ and $q=2$ and $r=t=s=1$. Then
 $\mathcal{A}(z)=\mathcal{B}(z)=3-(z+z^{-1})$ and $\mathcal{C}(z)=-1$. Then $P_1(z)=\mathcal{A}(z)\mathcal{B}(z)-\mathcal{C}(z)\mathcal{C}(z^{-1})=(z+z^{-1})^2-6(z+z^{-1})+8$.
Hence $P_1(z)=0 \Rightarrow \frac{1}{2}(z+z^{-1})=1$ or $2$.

\medskip

(1.1) The number of spanning trees. By Theorem \ref{th1}(1), we have
$$
\tau(\Gamma_1)=\frac{n}{2}\left|2 T_n\left(2\right)-2\right|.
$$

\medskip

(1.2) The arithmetic properties of $\tau(\Gamma_1)$.
By Theorem \ref{th2}(1), there exist two integer sequences $a_1(n)$ and $b_1(n)$ such that
\begin{eqnarray}
\tau(\Gamma_1)=\begin{cases}
na_1(n)^2, &n~is~odd,\\
6nb_1(n)^2 , &n~is~even.\\
\end{cases}
\nonumber
\end{eqnarray}

\medskip

(1.3) The asymptotics of $\tau(\Gamma_1)$. By Theorem \ref{th3}(1), we have $\tau(\Gamma_1)\sim \frac{n}{2}A^n$, where $A=2+\sqrt{3}$.

\medskip

(1.4) The generating function of $\tau(\Gamma_1)$. By Theorem \ref{th4}, we obtain
$$
F_1(x)=\sum_{n=1}^{\infty}\tau(\Gamma_1)x^n=\frac{u^2+u-3}{2(u-2)^2(u-1)},
$$
where $u=\frac{1}{2}(x+\frac{1}{x})$.
\medskip

\textit{(2)} \underline{The graph $\Gamma=BC(\mathbb{Z}_n;\{1,-1,\frac{n}{2}\},\emptyset,\{0\})$}. In this case, we have $\Gamma=\Gamma_2$ and $q=r=s=1$ and $t=0$. Then $\mathcal{A}(z)=3-(z+z^{-1})$ and $\mathcal{B}(z)=1$ and $\mathcal{C}(z)=-1$. Then $P_2(z)=(\mathcal{A}(z)+2)\mathcal{B}(z)-\mathcal{C}(z)\mathcal{C}(z^{-1})=4-(z+z^{-1})$,
$P_1(z)=\mathcal{A}(z)\mathcal{B}(z)-\mathcal{C}(z)\mathcal{C}(z^{-1})=2-(z+z^{-1})$. Hence $P_2(z)=0$ and $P_1(z)=0$ $ \Rightarrow \frac{1}{2}(z+z^{-1})=1$ or $2$.

\medskip

(2.1) The number of spanning trees. By Theorem \ref{th1}(2), we have
$$
\tau(\Gamma_2)=\frac{n}{4}\left|2T_\frac{n}{2}\left(2\right)+2\right|.
$$

\medskip

(2.2) The arithmetic properties of $\tau(\Gamma_2)$.
By Theorem \ref{th2}(2), there exist two integer sequences $a_2(n)$ and $b_2(n)$ such that
\begin{eqnarray}
\tau(\Gamma_2)=\begin{cases}
\frac{3n}{2}a_2(n)^2, &\frac{n}{2}~is~odd,\\
\frac{n}{4}b_2(n)^2 , &\frac{n}{2}~is~even.\\
\end{cases}
\nonumber
\end{eqnarray}

\medskip

(2.3) The asymptotics of $\tau(\Gamma_2)$. By Theorem \ref{th3}(2), we have $\tau(\Gamma_2)\sim \frac{n}{4}B^\frac{n}{2}$, where $B=2+\sqrt{3}$.

\medskip

(2.4) The generating function of $\tau(\Gamma_2)$. Since $n$ is even, we can replace $n$ with $2n$.
By Theorem \ref{th4}, we obtain
$$
F_2(x)=\sum_{n=1}^{\infty}\tau(\Gamma_2)x^n=\frac{3u^2-7u+5}{2(u-2)^2(u-1)},
$$
where $u=\frac{1}{2}(x+\frac{1}{x})$.
\medskip

\textit{(3)} \underline{The graph $\Gamma=BC(\mathbb{Z}_n;\{1,-1\},\{\frac{n}{2}\},\{0\})$}. In this case, we have $\Gamma=\Gamma_3$ and $q=r=s=1$ and $t=0$. Then $\mathcal{A}(z)=3-(z+z^{-1})$,
$\mathcal{B}(z)=1$ and $\mathcal{C}(z)=-1$. Then $P_3(z)=\mathcal{A}(z)(\mathcal{B}(z)+2)-\mathcal{C}(z)\mathcal{C}(z^{-1})=8-3(z+z^{-1})$,
$P_1(z)=\mathcal{A}(z)\mathcal{B}(z)-\mathcal{C}(z)\mathcal{C}(z^{-1})=2-(z+z^{-1})$. Hence $P_3(z)=0$ and $P_1(z)=0$ $ \Rightarrow \frac{1}{2}(z+z^{-1})=\frac{4}{3}$ or $1$.

\medskip

(3.1) The number of spanning trees. By Theorem \ref{th1}(3), we have
$$
\tau(\Gamma_3)=\frac{3^{\frac{n}{2}}n}{4}\left|2 T_\frac{n}{2}\left(\frac{4}{3}\right)+2\right|.
$$

\medskip

(3.2) The arithmetic properties of $\tau(\Gamma_3)$.
By Theorem \ref{th2}(3), there exist two integer sequences $a_3(n)$ and $b_3(n)$ such that
\begin{eqnarray}
\tau(\Gamma_3)=\begin{cases}
\frac{7n}{2}a_3(n)^2, &\frac{n}{2}~is~odd,\\
\frac{n}{4}b_3(n)^2 , &\frac{n}{2}~is~even.\\
\end{cases}
\nonumber
\end{eqnarray}

\medskip

(3.3) The asymptotics of $\tau(\Gamma_3)$. By Theorem \ref{th3}(3), we have $\tau(\Gamma_3)\sim  \frac{n}{4}C^\frac{n}{2}$, where $C=4+\sqrt{7}$.

\medskip

(3.4) The generating function of $\tau(\Gamma_3)$. Since $n$ is even, we can replace $n$ with $2n$.
By Theorem \ref{th4}, we obtain
$$
F_3(x)=\sum_{n=1}^{\infty}(\Gamma_3)x^n=\frac{45u^2-87u+43}{2(3u-4)^2(u-1)},
$$
where $u=\frac{1}{2}(3x+\frac{1}{3x})$.

\medskip

\textit{(4)} \underline{The graph $\Gamma=BC(\mathbb{Z}_n;\{1,-1,\frac{n}{2}\},\{\frac{n}{2}\},\{0\})$}. In this case, we have $\Gamma=\Gamma_4$ and $q=r=s=1$ and $t=0$. Then $\mathcal{A}(z)=3-(z+z^{-1})$,
$\mathcal{B}(z)=1$ and $\mathcal{C}(z)=-1$. Then $P_4(z)=(\mathcal{A}(z)+2)(\mathcal{B}(z)+2)-\mathcal{C}(z)\mathcal{C}(z^{-1})=14-3(z+z^{-1})$,
$P_1(z)=\mathcal{A}(z)\mathcal{B}(z)-\mathcal{C}(z)\mathcal{C}(z^{-1})=2-(z+z^{-1})$. Hence $P_4(z)=0$ and $P_1(z)=0$ $ \Rightarrow \frac{1}{2}(z+z^{-1})=\frac{7}{3}$ or $1$.

\medskip

(4.1) The number of spanning trees. By Theorem \ref{th1}(4), we have
$$
\tau_{\Gamma_4}(n)=\frac{3^{\frac{n}{2}}n}{4}\left|2 T_\frac{n}{2}\left(\frac{7}{3}\right)+2\right|.
$$

\medskip

(4.2) The arithmetic properties of $\tau(\Gamma_4)$.
By Theorem \ref{th2}(4), there exist two integer sequences $a_4(n)$ and $b_4(n)$ such that
\begin{eqnarray}
\tau(\Gamma_4)=\begin{cases}
\frac{5n}{4}a_4(n)^2, &\frac{n}{2}~is~odd,\\
\frac{n}{4}b_4(n)^2 , &\frac{n}{2}~is~even.\\
\end{cases}
\nonumber
\end{eqnarray}

\medskip

(4.3) The asymptotics of $\tau(\Gamma_4)$. By Theorem \ref{th3}(4), we have $\tau(\Gamma_4)\sim \frac{n}{4}D^\frac{n}{2}$, where $D=7+2\sqrt{10}$.

\medskip

(4.4) The generating function of $\tau(\Gamma_4)$. Since $n$ is even, we can replace $n$ with $2n$.
By Theorem \ref{th4}, we obtain
$$
F_4(x)=\sum_{n=1}^{\infty}\tau(\Gamma_4)x^n=\frac{72u^2-132u+76}{2(3u-7)^2(u-1)},
$$
where $u=\frac{1}{2}(3x+\frac{1}{3x})$.

\section*{Declaration of competing interests}
We declare that we have no conflict of interests to this work.

\section*{Data availability}
No data was used for the research described in the article.

 \end{document}